\documentclass[11pt]{article}
\usepackage{amssymb,amsmath,amsfonts,amsthm,mathrsfs}
\usepackage{graphicx,graphics,psfrag,epsfig,calrsfs,cancel}
\usepackage[colorlinks=true, pdfstartview=FitV, linkcolor=blue, citecolor=red, urlcolor=blue]{hyperref}

\usepackage{caption,subfig,color}
\usepackage{tikz}
\usetikzlibrary{calc}

\usepackage{geometry}
\newtheorem{theorem}{Theorem}[section]
\newtheorem{definition}[theorem]{Definition}

\newtheorem{assumption}[theorem]{Assumption}
\newtheorem{lemma}[theorem]{Lemma}

\numberwithin{equation}{section}
\numberwithin{figure}{section}

\begin{document}

\title{High order numerical schemes \\ for transport equations on bounded domains}

\author{Benjamin {\sc Boutin}\thanks{IRMAR (UMR CNRS 6625), Universit\'e de Rennes, Campus de Beaulieu, 35042 Rennes Cedex, France ; 
{\tt benjamin.boutin@univ-rennes1.fr}} $\,$ \& Thi Hoai Thuong {\sc Nguyen}\thanks{IRMAR (UMR CNRS 6625), Universit\'e de Rennes, Campus 
de Beaulieu, 35042 Rennes Cedex, France ; {\tt thi-hoai-thuong.nguyen@univ-rennes1.fr}} $\,$ \& Abraham {\sc Sylla}\thanks{Institut Denis Poisson 
CNRS UMR 7013, Universit\'e de Tours, Facult\'e des Sciences et Techniques, B\^atiment E2, Parc de Grandmont, 37200 Tours, France ; 
{\tt Abraham.Sylla@lmpt.univ-tours.fr}} \\
\& S\'ebastien {\sc Tran-Tien}\thanks{\'Ecole Normale Sup\'erieure de Lyon, 15 parvis Ren\'e Descartes, 
BP 7000, 69342 Lyon Cedex 07, France ; {\tt sebastien.tran-tien@ens-lyon.fr}} $\,$ \& Jean-Fran\c{c}ois {\sc Coulombel}\thanks{Institut de 
Math\'ematiques de Toulouse ; UMR5219, Universit\'e de Toulouse ; CNRS, F-31062 Toulouse Cedex 9, France ; 
{\tt jean-francois.coulombel@math.univ-toulouse.fr}. Research of all authors was supported by the ANR project NABUCO, ANR-17-CE40-0025.}
}
\maketitle

\begin{abstract}
This article is an account of the NABUCO project achieved during the summer camp CEMRACS 2019 devoted to geophysical fluids and 
gravity flows. The goal is to construct finite difference approximations of the transport equation with nonzero incoming boundary data that 
achieve the best possible convergence rate in the maximum norm. We construct, implement and analyze the so-called inverse Lax-Wendroff 
procedure at the incoming boundary. Optimal convergence rates are obtained by combining sharp stability estimates for extrapolation 
boundary conditions with numerical boundary layer expansions. We illustrate the results with the Lax-Wendroff and $O3$ schemes.
\end{abstract}

%\begin{resume}
%Cet article est un compte-rendu du projet NABUCO r\'ealis\'e au cours du CEMRACS 2019 qui \'etait consacr\'e aux fluides g\'eophysiques et 
%aux \'ecoulements gravitaires. Le but est de construire des approximations par diff\'erences finies de l'\'equation de transport avec une condition 
%d'entr\'ee non-nulle qui donnent lieu au meilleur taux de convergence possible dans la norme de la convergence uniforme. Nous construisons, 
%impl\'ementons et analysons la proc\'edure, dite de Lax-Wendroff inverse, pour les mailles d'entr\'ee. Des taux de convergence optimaux sont 
%obtenus en combinant des estimations de stabilit\'e fines pour les conditions d'extrapolation en sortie et des d\'eveloppements de type couche 
%limite. Nous illustrons les r\'esultats obtenus sur les sch\'emas de Lax-Wendroff et $O3$.
%\end{resume}

%%%%%%%%%%%%%%%%%%%%%%%%%%%%%%%%%
\section{Introduction}

%%%%%%%%%%%%%%%%%%%%%%%%%%%%%%%%%
\subsection{Context}

The goal of this article is to propose a high order numerical treatment of nonzero incoming boundary data for the advection equation. 
The methodology is developed here for the one-dimensional problem but it is our hope that the tools used below will be useful for higher 
dimensional problems. We are thus given a fixed constant velocity $a>0$, an interval length $L>0$ and we consider the (continuous) 
problem:
\begin{equation}
\label{TBE}
\begin{cases}
\partial_t u +a \, \partial_x u \, = \, 0 \, , & t \ge 0 \, , \quad x \in (0,L) \, ,\\
u(0,x) \, = \, f(x) \, ,& x \in (0,L) \, ,\\
u(t,0) \, = \, g(t) \, ,& t \ge 0 \, .
\end{cases}
\end{equation}
The requirements on the initial and boundary data, namely $f$ and $g$, will be made precise below. The solution to \eqref{TBE} is given 
by the method of characteristics, which yields the formula:
\begin{equation}
\label{EXS}
\forall \, (t,x) \in \mathbb{R}^+ \times (0,L) \, ,\quad u(t,x) \, = \, \begin{cases}
f \left( x-a \, t \right) \, ,& \text{\rm if } x \ge a \, t \, ,\\
g \left(t-\dfrac{x}{a}\right) \, ,& \text{\rm if } x \le a \, t \, .
\end{cases}
\end{equation}
The question we address is how to construct high order numerical approximations of the solution \eqref{EXS} to \eqref{TBE} by means of 
(explicit) finite difference approximations. This problem has been addressed in \cite{CL18} in the case of \emph{zero} incoming boundary 
data (that is, $g=0$ in \eqref{TBE}). The focus in \cite{CL18} is on the outflow boundary ($x=L$ here since $a$ is positive), for which 
\emph{extrapolation} numerical boundary conditions are analyzed. Fortunately for us, a large part of the analysis in \cite{CL18} can be used 
here as a black box and we therefore focus on the incoming boundary. To motivate the analysis of this paper, let us present a very simple 
-though illuminating- example for which we just need to introduce the basic notations that will be used throughout this article. In all what 
follows, we consider a positive integer $J$, that is meant to be large, and define the space step $\Delta x$ and the grid points 
$(x_j)_{j \in \mathbb{Z}}$ by
$$
\Delta x \, := \, \dfrac{L}{J} \, ,\quad x_j \, := \, j \, \Delta x \quad (j \in \mathbb{Z}) \, .
$$
The interval $(0,L)$ corresponds to the cells $(x_{j-1},x_j)$ with $j=1,\dots,J$, but considering the whole real line $\{ j \in \mathbb{Z} \}$ will 
be useful in some parts of the analysis. The time step $\Delta t$ is then defined as $\Delta t := \lambda \, \Delta x$, where $\lambda >0$ is 
a constant that is fixed so that assumption \ref{ass:consistencystability} below is satisfied. We use from now on the notation $t^n := n \, 
\Delta t$, $n \in \mathbb{N}$; the quantity $u^n_j$ will play the role of an approximation for the solution $u$ to \eqref{TBE} at the time $t^n$ 
on the cell $(x_{j-1},x_j)$.

We now examine an example where the exact solution to \eqref{TBE} is approximated by means of the Lax-Wendroff scheme. The 
approximation reads:
\begin{equation}
\label{LW}
u_j^{n+1} \, = \, u_j^n -\dfrac{\lambda \, a}{2} \, (u_{j+1}^n -u_{j-1}^n) +\dfrac{(\lambda \, a)^2}{2} \, (u_{j+1}^n -2\, u_j^n +u_{j-1}^n) \, ,\quad 
n \in \mathbb{N} \, , \quad j=1,\dots,J \, ,
\end{equation}
where we recall that $\lambda=\Delta t/\Delta x$ is a fixed constant and $a>0$ is the transport velocity in \eqref{TBE}. The initial condition 
for \eqref{LW} is defined, for instance, by computing the cell averages of the initial condition $f$ in \eqref{TBE}, namely:
\begin{equation}
\label{initialdata}
\forall \, j=1,\dots,J \, ,\quad u_j^0 \, := \, \dfrac{1}{\Delta x} \, \int_{x_{j-1}}^{x_j} f(x) \, {\rm d}x \, .
\end{equation}
Without any boundary, the Lax-Wendroff scheme is a second order approximation to the transport equation \cite{gko}. We would like, of course, 
to maintain the second order accuracy property when implementing \eqref{LW} on an interval. This implementation, however, requires, at each 
time iteration $n$, the definition of the boundary (or \emph{ghost cell}) values $u_0^n$ and $u_{J+1}^n$.  At the outflow boundary, we prescribe 
an extrapolation condition \cite{kreissproc,goldberg}, the significance of which will be thoroughly justified in the next sections:
\begin{equation}
\label{LW-outflow}
u_{J+1}^n \, = \, 2 \, u_J^n -u_{J-1}^n \, ,\quad n \in \mathbb{N} \, .
\end{equation}
Combining \eqref{LW} with \eqref{LW-outflow}, the last interior cell value $u_J^n$ obeys the induction formula:
$$
u_J^{n+1} \, = \, u_J^n -\lambda \, a \, (u_J^n -u_{J-1}^n) \, ,\quad n \in \mathbb{N} \, ,
$$
which is nothing but the upwind scheme. It then only remains to determine the inflow numerical boundary condition $u_0^n$. Since we wish to 
approximate the exact solution to \eqref{TBE} and $u_0^n$ is meant, at least, to approximate the trace $u(t^n,0)$, it seems reasonable at first 
sight to prescribe the Dirichlet boundary condition:
\begin{equation}
\label{LW-inflow}
u_0^n \, = \, g(t^n) \, ,\quad n \in \mathbb{N} \, .
\end{equation}
In the case of \emph{zero} incoming boundary data ($g=0$), and for any sufficiently smooth initial condition $f$ that is ``flat'' at the incoming 
boundary, the main result of \cite{CL18} shows that the above numerical scheme \eqref{LW}, \eqref{initialdata}, \eqref{LW-outflow}, \eqref{LW-inflow} 
converges towards the exact solution to \eqref{TBE} with a rate of convergence $3/2$ in the maximum norm. Numerical simulations even predict 
that the rate of convergence should be $2$, or at least close to $2$, for smooth initial data. However, implementing the above numerical 
scheme\footnote{One can choose for instance $a=1$, $\lambda=5/6$, $L=6$, $f(x)=\sin (x)$, $g(t)=-\sin (t)$ and increase the integer $J$ 
geometrically.} quickly shows that the rate of convergence falls down to $1$ when $g$ is nonzero and satisfies the compatibility 
conditions\footnote{The rate of convergence could be even smaller than $1$ when the compatibility conditions are not satisfied but that would 
just reflect the fact that the exact solution \eqref{EXS} is not smooth (for instance, not even continuous if $f(0) \neq g(0)$).} described hereafter 
with the initial condition $f$.

Our goal is to provide with a thorough treatment of nonzero incoming boundary data and to design numerical boundary conditions that recover 
the optimal rate of convergence in the maximum norm (at least, the same rate of convergence as the one in \cite{CL18} for zero boundary data). 
The strategy is not new and is now referred to as the \emph{inverse Lax-Wendroff method}. It consists, as detailed below, in writing Taylor expansions 
with respect to the space variable $x$ close to the incoming boundary and then using the advection equation \eqref{TBE} to substitute the normal 
derivatives $\partial_x^m u(t,0)$ for tangential derivatives $\partial_t^m u(t,0)$, the latter being computed thanks to the boundary condition in 
\eqref{TBE}. This strategy is available when the boundary is non-characteristic \cite{BS}.

The inverse Lax-Wendroff method is a general strategy that has been followed in various directions. We refer for instance to 
\cite{TS10,ShuTan,FY13,VS15,DDJ} for various implementations related to either hyperbolic or kinetic partial differential equations. In these 
works, most of the time, the incoming numerical boundary condition prescribes the ghost cell value $u_0^n$ in terms of the boundary datum 
$g$ but also of \emph{interior cell} values $u_j^n$ with $j \ge 1$. This is the reason why \emph{stability} is a real issue in these works, see for 
instance the discussion in \cite[Section 4]{VS15}, and many rigorous justifications are still open. We develop here a simplified version of some 
of those previously proposed boundary treatments, but we rigorously justify the convergence with an (almost) optimal rate of convergence. As 
in \cite{CL18}, the key ingredient in our analysis is an \emph{unconditional} stability result for the Dirichlet boundary conditions which dates 
back to \cite{goldberg-tadmor1,goldberg-tadmor2}, see an alternative proof in \cite{jfcag}.

%%%%%%%%%%%%%%%%%%%%%%%%%%%%%%%%%
\subsection{The inverse Lax-Wendroff method}

We first fix from now on some notations. In all this article, we are given some fixed integers $p,r \in \mathbb{N}$ and consider an explicit two time 
step approximation for the solution to \eqref{TBE}:
\begin{equation}
\label{eq:scheme}
u_j^{n+1} \, = \, \sum_{\ell=-r}^p a_\ell \, u_{j+\ell}^n \, ,\quad n \in \mathbb{N} \, , \quad j=1,\dots,J \, .
\end{equation}
In \eqref{eq:scheme}, the numbers $a_{-r},\dots,a_p$ are defined in terms of the parameter $\lambda$ and of the velocity $a$ (see, for instance, 
\eqref{LW} for which $p=r=1$). These numbers are fixed, which means that \eqref{eq:scheme} is linear with respect to $(u_j^n)$. For simplicity, 
we follow \cite{CL18} and choose as initial data for \eqref{eq:scheme} the cell averages of the initial condition $f$ in \eqref{TBE}. This means 
that the vector $(u_1^0,\dots,u_J^0)$ is defined by \eqref{initialdata}. For \eqref{eq:scheme} to define inductively (with respect to $n$) the vector 
$(u_1^n,\dots,u_J^n)$, we need to prescribe the ghost cell values $u_{1-r}^n,\dots,u_0^n$ and $u_{J+1}^n,\dots,u_{J+p}^n$. As explained above, 
we focus here on the inflow boundary and we therefore follow the extrapolation boundary treatment of \cite{CL18} for the outflow boundary. Namely, 
if we define the finite difference operator $D_-$ as:
$$
(D_- v)_j \, := \, v_j -v_{j-1} \, ,
$$
and its iterates $D_-^m$ accordingly, we choose from now on an extrapolation order $k_b \in \mathbb{N}$ for the outflow boundary and prescribe:
\begin{equation}
\label{eq:bcNeumann}
(D_-^{k_b} u^n)_{J+\ell} \, = \, 0 \, ,\quad n \in \mathbb{N} \, , \quad \ell=1,\dots,p \, .
\end{equation}
The example \eqref{LW-outflow} corresponds to $k_b=2$ (recall $p=1$ for the Lax-Wendroff scheme so there is only one ghost cell). It now remains 
to prescribe the inflow values $u_{1-r}^n,\dots,u_0^n$. Unlike some previous works, we are going to prescribe Dirichlet boundary conditions, meaning 
for instance that the value $u_0^n$ will be determined in terms of the boundary datum $g$ only. Let us assume for a while that $u_j^n$ is a second 
order approximation of $u(t^n,(x_{j-1}+x_j)/2)$ where $u$ is the exact solution \eqref{EXS} of the continuous problem \eqref{TBE}. Then we 
\emph{formally} have:
$$
u_0^n \approx u \left( t^n,-\dfrac{\Delta x}{2} \right) \approx u(t^n,0) -\dfrac{\Delta x}{2} \, \partial_x u(t^n,0) \, ,
$$
where $\approx$ means ``equal up to $O(\Delta x^2)$'', and we then use \eqref{TBE} to get:
$$
u_0^n \approx u(t^n,0) +\dfrac{\Delta x}{2 \, a} \, \partial_t u(t^n,0) \, = \, g(t^n) +\dfrac{\Delta x}{2 \, a} \, g'(t^n) \, .
$$
The last term $\Delta x/(2\, a) \, g'(t^n)$ in the previous equality is precisely the correction that is required to recover the second order accuracy when 
dealing with the Lax-Wendroff scheme (compare with \eqref{LW-inflow}). More generally speaking, we could have pushed further the above Taylor 
expansion and obtained as a final (formal !) result that $u_0^n$ should be ``close'' (whatever that means !) to some quantity of the form:
$$
\sum_{\kappa=0}^K \dfrac{\Delta x^\kappa}{\kappa \, ! \, a^\kappa} \, \alpha_\kappa \, g^{(\kappa)}(t^n) \, ,
$$ 
where $K$ is a truncation order and $\alpha_0,\dots,\alpha_K$ are numerical constants.

The general form of the Dirichlet boundary conditions that we consider below is:
$$
u_\ell^n \, = \, \sum_{\kappa=0}^K \dfrac{\Delta x^\kappa}{\kappa \, ! \, a^\kappa} \, \alpha_{\kappa,\ell} \, g^{(\kappa)}(t^n) 
\, ,\quad n \in \mathbb{N} \, , \quad \ell=1-r,\dots,0 \, ,
$$
where the $\alpha_{\kappa,\ell}$'s are numerical constants which will play a role (together with the truncation order $K$) in the consistency analysis. 
There are two main choices which we discuss in this article. The first one is given in \cite{TS10,VS15}:
$$
\alpha_{\kappa,\ell} \, := \, \left( \dfrac{1}{2}-\ell \right)^\kappa \, ,\quad \kappa \in \mathbb{N} \, , \quad \ell=1-r,\dots,0 \, ,
$$
and is relevant if $u_j^n$ is eventually compared in the convergence analysis with $u(t^n,(x_{j-1}+x_j)/2)$, $u$ being the exact solution \eqref{EXS}. 
The other possible choice we advocate is:
\begin{equation}
\label{alpha}
\alpha_{\kappa,\ell} \, := \, \dfrac{(-1)^\kappa}{\kappa+1} \, \left( \ell^{\kappa+1} -(\ell-1)^{\kappa+1} \right) \, ,\quad 
\kappa \in \mathbb{N} \, , \quad \ell=1-r,\dots,0 \, ,
\end{equation}
and is relevant if $u_j^n$ is eventually compared (as in our main result, which is Theorem \ref{mainthm} below)  in the convergence analysis with the 
average of $u(t^n,\cdot)$ on the cell $(x_{j-1},x_j)$. The truncation order $K$ is discussed with our main result in the following paragraph.

%%%%%%%%%%%%%%%%%%%%%%%%%%%%%%%%%
\subsection{Results}

We assume that the approximation \eqref{eq:scheme} is consistent with the transport operator and that it defines a stable procedure on $\ell^2(\mathbb{Z})$.

\begin{assumption}[Consistency and stability without any boundary]
\label{ass:consistencystability}
The coefficients $a_{-r},... ,a_p$ in \eqref{eq:scheme} satisfy $a_{-r} \, a_p\neq 0$ (normalization), and for some integer $k \ge 1$, there holds:
\begin{align}
\forall \, m=0,\dots,k \, ,\quad \sum_{\ell=-r}^p \ell^m \, a_\ell &\, = \, (-\lambda \, a)^m \, ,\quad \text{\rm (consistency of order $k$),} \label{Assu1} \\
\sup_{\theta \in [0,2 \pi]} \left\vert \sum_{\ell=-r}^p a_\ell \, e^{i \, \ell \, \theta} \right\vert & \, \le \, 1 \, ,\quad 
\text{\rm ($\ell^2$-stability on $\mathbb{Z}$).} \label{Assu2}
\end{align}
\end{assumption}

\noindent For the Lax-Wendroff scheme \eqref{LW}, we have $p=r=1$ if $\lambda \, a \neq 1$, the integer $k$ equals $2$, and \eqref{Assu2} holds if 
and only if $\lambda \, a \le 1$. In Theorem \ref{mainthm} below and all what follows, the velocity $a>0$, the length $L>0$, the parameter $\lambda = 
\Delta t/\Delta x$ and the extrapolation order $k_b \in \mathbb{N}$ at the outflow boundary are given. Subsequent constants may depend on them. The 
integer $k \ge 1$ is also fixed such that assumption \ref{ass:consistencystability} holds. We consider the initial condition \eqref{initialdata} and its evolution 
by the numerical scheme \eqref{eq:scheme}, \eqref{eq:bcNeumann}, the inflow ghost cell values being given by:
\begin{equation}
\label{eq:bcDirichlet}
u_\ell^n \, = \, \sum_{\kappa=0}^{k-1} \dfrac{\Delta x^\kappa}{(\kappa+1) \, ! \, (-a)^\kappa} \, \big( \ell^{\kappa+1} -(\ell-1)^{\kappa+1} \big) \, g^{(\kappa)}(t^n) 
\, ,\quad n \in \mathbb{N} \, , \quad \ell=1-r,\dots,0 \, .
\end{equation}
Of course, prescribing \eqref{eq:bcDirichlet} is meaningful only if $g$ is sufficiently smooth (say, $g \in \mathcal{C}^{k-1}$). One could push further the 
Taylor expansion in \eqref{eq:bcDirichlet} and consider higher order correctors but it would require further smoothness on $g$ and it would eventually 
not improve our convergence result below, so fixing the truncation order $K=k-1$ seems to be the most convenient choice. Our main convergence result 
is the extension of the main result in \cite{CL18} to the case of nonzero boundary data.

\begin{theorem}[Main convergence result]
\label{mainthm}
Under assumption \ref{ass:consistencystability}, there exists a constant $C>0$ such that for any final time $T \ge 1$, any integer $J\in\mathbb{N}^*$, 
any data $f \in H^{k+1}((0,L))$ and $g \in H^{k+1}((0,T))$ satisfying the compatibility requirements at $t=x=0$:
$$
\forall \, m=0,\dots,k \, ,\quad f^{(m)}(0) \, = \, (-a)^{-m} \, g^{(m)}(0) \, ,
$$
the solution $(u_j^n)$ to \eqref{initialdata}, \eqref{eq:scheme}, \eqref{eq:bcNeumann}, \eqref{eq:bcDirichlet} satisfies:
\begin{equation}
\label{eq:convergenceresult}
\sup_{0 \le n \le T/\Delta t} \, \sup_{1 \le j \le J} \left\vert u_j^n -\dfrac{1}{\Delta x} \, \int_{x_{j-1}}^{x_j} u(t^n,x) \, {\rm d}x \right\vert 
\le C \, T \, e^{C\, T/L} \, \Delta x^{\min (k,k_b)-1/2} \, \big( \Vert f \Vert_{H^{k+1}((0,L))} +\Vert g \Vert_{H^{k+1}((0,T))} \big),
\end{equation}
with $u$ the exact solution to \eqref{TBE}, whose expression is given by \eqref{EXS}.
\end{theorem}

\noindent Actually, the constant $C$ in \eqref{eq:convergenceresult} is independent of $L \ge 1$, which is consistent with the convergence result we shall 
prove below for the half-space problem on $\mathbb{R}^+$ with inflow at $x=0$. As in \cite{CL18}, the loss of $1/2$ in the rate of convergence of Theorem 
\ref{mainthm} looks somehow artificial and is mostly a matter of passing from the $\ell^\infty_n \ell^2_j$ topology to $\ell^\infty_{n,j}$. Our next result examines 
a situation where the optimal convergence rate $\min (k,k_b)$ can be obtained. In order to simplify (and shorten) the proof of Theorem \ref{mainthm'}, we 
only examine here the case of a half-space with extrapolation outflow conditions. The extension of the techniques to the case of an interval is left to the 
interested reader.

\begin{theorem}[Optimal rate of convergence for the outflow problem]
\label{mainthm'}
Under assumption \ref{ass:consistencystability} and under the additional assumption \ref{assumroot} stated hereafter, there exists a constant $C>0$ such 
that for any final time $T \ge 1$, any integer $J\in\mathbb{N}^*$, any data $f \in H^{k+1}((-\infty,L))$, the solution to the scheme:
\begin{equation}
\label{neumannsortie}
\left\{
\begin{aligned}
 &u_j^0 \, = \, \dfrac{1}{\Delta x} \, \int_{x_{j-1}}^{x_j} f(x) \, {\rm d}x \, , & & j \le J \, ,\\
 &(D_-^{k_b}u^n)_{J+\ell} \, = \, 0 \, , & & 0 \le n \le T/\Delta t \, , \quad \ell=1,\ldots,p \, ,\\
 &u_j^{n+1} \, = \, \sum_{\ell=-r}^p a_\ell \, u_{j+\ell}^n \, ,& & 0 \le n \le T/\Delta t -1 \, ,\quad j \le J \, ,
\end{aligned}
\right.
\end{equation}
satisfies the error estimate
$$
\sup_{0 \le n \le T/\Delta t} \, \sup_{j \le J} \, \left\vert u_j^n -\dfrac{1}{\Delta x} \, \int_{x_{j-1}}^{x_j} f(x-a\, t^n) \, {\rm d}x \right\vert 
\, \le C \, T \, \Delta x^{k_b} \, \Vert f \Vert_{H^{k+1}((0,L))} \, ,
$$
as long as $k_b<k$.
\end{theorem}

In other words, the technical assumption \ref{assumroot} hereafter, which is verified on many examples such as the Lax-Wendroff and $O3$ schemes, 
allows to recover the optimal rate $k_b=\min (k_b,k)$ in the case $k_b<k$. Of course, one would also like to improve the rate $\min (k_b,k)-1/2$ in the 
case $k_b=k$, which is clearly the most natural choice. However, in that case, both the interior and boundary consistency errors scale like $\Delta x^k$ 
and, in the framework of assumption \ref{ass:consistencystability}, stability in the interior domain is available only in the $\ell^2_j$ topology, so it is quite 
difficult to derive the convergence rate $k$ in the $\ell^\infty_j$ topology. Theorem \ref{mainthm'} already indicates that combining the approach of 
\cite{CL18} with other techniques (here, boundary layer expansions) may improve some results. We hope to deal with the case $k_b=k$ in the future.

%%%%%%%%%%%%%%%%%%%%%%%%%%%%%%%%%
\section{Convergence analysis for the inverse Lax-Wendroff method}

This Section is devoted to the proof of Theorem \ref{mainthm}. In order to shorten the exposition, we shall use some results of \cite{CL18} as a black box 
and we refer the interested reader to \cite{CL18} for more details. Following \cite{CL18}, we shall prove Theorem \ref{mainthm} by using a stability estimate 
for \eqref{eq:scheme}, \eqref{eq:bcNeumann}, \eqref{eq:bcDirichlet} and a superposition argument, which amounts to considering separately two half-space 
problems: one in which there is only inflow at $x=0$, and one for which there is only outflow at $x=L$. The novelty here is the nonzero inflow source term so 
we first deal with that case.

%%%%%%%%%%%%%%%%%%%%%%%%%%%%%%%%%
\subsection{Convergence analysis on a half-line for the inflow problem}

We focus here on the inflow source term, and therefore start by proving the main convergence estimate that is the new ingredient for the proof of Theorem 
\ref{mainthm}.

\begin{theorem}[Convergence estimate for the inflow problem]
\label{thm-cv-dirichlet}
Under assumption \ref{ass:consistencystability}, there exists a constant $C>0$ such that for any final time $T \ge 1$, for any $J\in\mathbb{N}^*$, for any initial 
condition $f \in H^{k+1}((0,+\infty))$ and boundary source term $g \in H^{k+1}((0,T))$ satisfying the compatibility conditions:
\begin{equation}
\label{compatiblite}
\forall \, m=0,\dots,k \, ,\quad f^{(m)}(0) \, = \, (-a)^{-m} \, g^{(m)}(0) \, ,
\end{equation}
the solution $(u_j^n)_{j \ge 1-r,n \in\mathbb{N}}$ to the numerical scheme:
\begin{align}
\label{demi-espace-1}
\begin{cases}
u_j^0 \, = \, \dfrac{1}{\Delta x} \, {\displaystyle \int_{x_{j-1}}^{x_j}} f(x) \, {\rm d}x \, ,& j \ge 1 \, ,\\
u^n_\ell \, = \, {\displaystyle \sum_{\kappa=0}^{k-1}} \dfrac{\Delta x^\kappa}{(\kappa+1) \, ! \, (-a)^\kappa} \, \big( \ell^{\kappa+1} -(\ell-1)^{\kappa+1} \big) 
\, g^{(\kappa)}(t^n)  \, ,& 0 \le n \le T/\Delta t \, , \quad \ell=1-r,\dots,0 \, ,\\
u_j^{n+1} ={\displaystyle \sum_{\ell=-r}^p} a_\ell \, u_{j+\ell}^n \, ,& 0 \le n \le T/\Delta t-1 \, , \quad j \ge 1 \, ,
\end{cases}
\end{align}
satisfies:
$$
\sup_{0 \le n \le T/\Delta t} \, \left( \, \sum_{j \ge 1} \Delta x \, \left( u_j^n -\dfrac{1}{\Delta x} \, \int_{x_{j-1}}^{x_j} u(t^n,x) \, {\rm d}x \right)^2 \, \right)^{1/2} 
\, \le \, C \, T \, \, \Delta x^k \, \big( \Vert f \Vert_{H^{k+1}((0,+\infty))} +\Vert g \Vert_{H^{k+1}((0,T))} \big) \, ,
$$
where $u$ is the exact solution to the half-line transport problem:
\begin{equation}
\label{transport'}
\begin{cases}
\partial_t u +a \, \partial_x u \, = \, 0 \, , & t \in (0,T) \, , \, x \ge 0 \, ,\\
u(0,x) \, = \, f(x) \, ,& x \ge 0 \, ,\\
u(t,0) \, = \, g(t) \, ,& t \in (0,T) \, .
\end{cases}
\end{equation}
\end{theorem}

\begin{proof}
For convenience, we first extend $g$ into a function $g_\flat \in H^{k+1}((0,+\infty))$ and then define:
\begin{equation*}
\forall \, x \in \mathbb{R} \, ,\quad f_\sharp (x) \, := \, \begin{cases}
f(x) \, , & \text{\rm if $x>0$,} \\
g_\flat (-x/a) \, , & \text{\rm if $x<0$.}
\end{cases}
\end{equation*}
Since $f$ and $g$ satisfy the compatibility conditions \eqref{compatiblite}, we have $f_\sharp \in H^{k+1}(\mathbb{R})$, and the exact solution $u$ to 
\eqref{transport'} is given by:
$$
\forall \, (t,x) \in [0,T] \times (0,+\infty) \, ,\quad u(t,x) \, = \, f_\sharp (x-a\, t) \, .
$$
Let us now define:
$$
\forall \, j \in \mathbb{Z} \, ,\quad \forall \, n \in \mathbb{N} \, ,\quad w_j^n \, := \, \dfrac{1}{\Delta x} \, \int_{x_{j-1}}^{x_j} f_\sharp(x-a\, t^n) \, {\rm d}x \, ,
$$
which corresponds to the cell average of the exact solution to \eqref{transport'}. With $(u_j^n)_{j \ge 1-r,0 \le n \le T/\Delta t}$ the solution to the numerical 
scheme \eqref{demi-espace-1}, we define the error $\varepsilon_j^n :=u_j^n -w_j^n$, that is a solution to:
\begin{equation}
\label{demi-espace-erreur}
\begin{cases}
\varepsilon_j^0 \, = \, 0 \, ,& j \ge 1 \, ,\\
\varepsilon^n_\ell \, = \, u^n_\ell-w^n_\ell  \, ,& 0 \le n \le T/\Delta t \, , \quad \ell=1-r,\dots,0 \, ,\\
\varepsilon_j^{n+1} ={\displaystyle \sum_{\ell=-r}^p} a_\ell \, \varepsilon_{j+\ell}^n +\Delta t \, e_j^{n+1} \, ,& 0 \le n \le T/\Delta t-1 \, , \quad j \ge 1 \, .
\end{cases}
\end{equation}

The interior consistency error $(e_j^{n+1})_{j \ge 1,0 \le n \le T/\Delta t-1}$ is easily estimated by means of the Cauchy-Schwarz inequality and Fourier analysis:
\begin{align*}
\sum_{j \ge 1} \Delta x \, (e_j^{n+1})^2 \, = & \, \dfrac{\Delta x}{\Delta t^2} \, \sum_{j \ge 1} \left( w_j^{n+1} -\sum_{\ell=-r}^p a_\ell \, w_{j+\ell}^n \right)^2 \\
= & \, \dfrac{1}{\Delta x \, \Delta t^2} \, \sum_{j \ge 1} \left( 
\int_{x_{j-1}}^{x_j} \left( f_\sharp(x-a\, t^n -a\, \Delta t) -\sum_{\ell=-r}^p a_\ell \, f_\sharp(x-a\, t^n+\ell \, \Delta x) \right) \, {\rm d}x \right)^2 \\
\le & \, \dfrac{1}{\Delta t^2} \, \int_{\mathbb{R}} 
\left( f_\sharp(x-a\, t^n -a\, \Delta t) -\sum_{\ell=-r}^p a_\ell \, f_\sharp(x-a\, t^n+\ell \, \Delta x) \right)^2 \, {\rm d}x \\
\le & \, \dfrac{1}{2 \, \pi \, \Delta t^2} \, \int_{\mathbb{R}} 
\Big| {\rm e}^{-i \, a \, \lambda \, \Delta x \, \xi} - \sum_{\ell=-r}^p a_\ell \,  {\rm e}^{i \, \ell \, \Delta x \, \xi} \Big|^2 \, \big| \widehat{f_\sharp}(\xi) \big|^2 \, {\rm d}\xi 
\, \le \, C \, \Delta x^{2 \, k} \, \int_{\mathbb{R}} \xi^{2\, (k+1)} \, \big| \widehat{f_\sharp}(\xi) \big|^2 \, {\rm d}\xi \, ,
\end{align*}
where the final inequality comes from assumption \ref{ass:consistencystability} and the fact that the ratio $\Delta t/\Delta x$ is constant. Going back to the 
definition of $f_\sharp$, we have obtained the bound:
\begin{equation}
\label{estimerreur-1}
\sup_{0 \le n \le T/\Delta t-1} \, \left( \sum_{j \ge 1} \Delta x \, (e_j^{n+1})^2 \right)^{1/2} \, \le \, C \, \Delta x^k \, 
\big( \Vert f \Vert_{H^{k+1}((0,+\infty))} +\Vert g \Vert_{H^{k+1}((0,T))} \big) \, ,
\end{equation}
for some constant $C$ that is independent of the final time $T \ge 1$ and the data $f$ and $g$.

We now turn to the boundary errors in \eqref{demi-espace-erreur}, and wish to estimate the following quantities:
$$
\sum_{0 \le n \le T/\Delta t-1} \Delta t \, \big( u_\ell^n-w_\ell^n \big)^2 \, ,\quad \ell=1-r,\dots,0 \, .
$$
Let us consider an integer $n$ such that $0 \le n \le T/\Delta t-1$. From the definition of $w_\ell^n$, $\ell \le 0$, we have:
\begin{align*}
u_\ell^n-w_\ell^n \, =& \, 
\sum_{\kappa=0}^{k-1} \dfrac{\Delta x^\kappa}{(\kappa+1) \, ! \, (-a)^\kappa} \, \big( \ell^{\kappa+1} -(\ell-1)^{\kappa+1} \big) \, g^{(\kappa)}(t^n) 
-\dfrac{1}{\Delta x} \, \int_{x_{\ell-1}}^{x_\ell} g_\flat(t^n-x/a) \, {\rm d}x \\
=& \, -\dfrac{(-a)^{-k}}{\Delta x} \, 
\int_{x_{\ell-1}}^{x_\ell} x^k \, \int_0^1 \dfrac{y^{k-1}}{(k-1) \, !} \, g_\flat^{(k)} \left( t^n -\dfrac{x \, y}{a} \right) \, {\rm d}y \, {\rm d}x \, ,
\end{align*}
where we have used the Taylor formula\footnote{This is precisely at this point of the analysis that the definition of the coefficients $\alpha_{\kappa,\ell}$ 
in the inverse Lax-Wendroff method arises. Our choice in \eqref{eq:bcDirichlet} is motivated by the fact that we compare the numerical solution with the 
cell average of the exact solution.}. By the Cauchy-Schwarz inequality, we get:
$$
\big( u_\ell^n-w_\ell^n \big)^2 \, \le \, \dfrac{C}{\Delta x} \, \int_{x_{\ell-1}}^{x_\ell} \int_0^1 x^{2\, k} y^{2\, (k-1)} \, 
g_\flat^{(k)} \left( t^n -\dfrac{x \, y}{a} \right)^2 \, {\rm d}y \, {\rm d}x \, ,
$$
and we now apply the change of variables $(x,y) \rightarrow (x\, y,x)$ to get:
$$
\big( u_\ell^n-w_\ell^n \big)^2 \, \le \, \int_{x_{\ell-1}}^{x_\ell} \left( \int_v^0 |v| \, |u|^{2\, (k-1)} \, g_\flat^{(k)} \left( t^n -\dfrac{u}{a} \right)^2 \, {\rm d}u 
\right) \, {\rm d}v \, .
$$
Restricting to $\ell=1-r,\dots,0$, we have:
$$
\sum_{\ell=1-r}^0 \big( u_\ell^n-w_\ell^n \big)^2 \, \le \, C \, \Delta x^{2\, k-1} \, \int_0^{r\, \Delta x/a} g_\flat^{(k)} (t^n+\tau)^2 \, {\rm d}\tau \, .
$$
Summing now with respect to $n$, we end up with the estimate:
\begin{equation}
\label{estimerreur-2}
\sum_{0 \le n \le T/\Delta t-1} \Delta t \, \sum_{\ell=1-r}^0 (\varepsilon_\ell^n)^2 \, \le \, C \, \Delta x^{2\, k} \, \int_0^{+\infty} g_\flat^{(k)} (t)^2 \, {\rm d}t 
\, \le \, C \, \Delta x^{2\, k} \, \Vert g \Vert_{H^k((0,T))}^2 \, ,
\end{equation}
for some constant $C$ that is independent of the final time $T \ge 1$ and the data $f$ and $g$.

We now apply the main stability estimate for the error problem \eqref{demi-espace-erreur}, for which we refer to the seminal papers 
\cite{goldberg-tadmor1,goldberg-tadmor2} and to the more recent works \cite{jfcag,CL18}:
\begin{multline*}
\sup_{0 \le n \le T/\Delta t} \left( \sum_{j \ge 1} \Delta x \, (\varepsilon_j^n)^2 \right)^{1/2} \, \le \, C \, \left\{ 
T \, \sup_{1 \le n \le T/\Delta t} \, \left( \sum_{j \ge 1} \Delta x \, (e_j^n)^2 \right)^{1/2} \right. \\
\left. + \left( \sum_{0 \le n \le T/\Delta t-1} \Delta t \, \sum_{\ell=1-r}^0 (\varepsilon_\ell^n)^2 \right)^{1/2} \right\} \, .
\end{multline*}
The conclusion of Theorem \ref{thm-cv-dirichlet} then comes from the combination of the estimates \eqref{estimerreur-1} and \eqref{estimerreur-2}.
\end{proof}

%%%%%%%%%%%%%%%%%%%%%%%%%%%%%%%%%
\subsection{Convergence estimate for the outflow problem}

We recall the convergence estimate obtained in \cite{CL18} for the complementary half-space problem where extrapolation conditions are prescribed at the 
outflow boundary. We refer the reader to \cite{CL18} for the details.

\begin{theorem}[Convergence estimate for the outflow problem \cite{CL18}]
\label{thm-cv-neumann}
Under assumption \ref{ass:consistencystability}, there exists a constant $C>0$ such that for any final time $T \ge 1$, for any $J\in\mathbb{N}^*$ and for any 
initial condition $f \in H^{k+1}((-\infty,L))$, the solution $(u_j^n)_{j \le J+p,0 \le n \le T/\Delta t}$ to the numerical scheme:
\begin{equation*}
\begin{cases}
u_j^0 \, = \, \dfrac{1}{\Delta x} \, {\displaystyle \int_{x_{j-1}}^{x_j}} f(x) \, {\rm d}x \, ,& j \le J \, ,\\
(D_{-}^{k_b}u^n)_{J+\ell} \, = \, 0 \, ,& 0 \le n \le T/\Delta t \, , \quad \ell=1,\dots,p \, ,\\
u_j^{n+1} ={\displaystyle \sum_{\ell=-r}^p} a_\ell \, u_{j+\ell}^n \, ,& 0 \le n \le T/\Delta t-1 \, , \quad j \le J \, ,
\end{cases}
\end{equation*}
satisfies:
$$
\sup_{0 \le n \le T/\Delta t} \, \left( \, \sum_{j \le J} \Delta x \, \left( u_j^n -\dfrac{1}{\Delta x} \, \int_{x_{j-1}}^{x_j} f(x-a \, t^n) \, {\rm d}x \right)^2 \, \right)^{1/2} 
\, \le \, C \, T \, \Delta x^{\min (k,k_b)} \, \Vert f \Vert_{H^{k+1}((-\infty,L))} \, .
$$
\end{theorem}

%%%%%%%%%%%%%%%%%%%%%%%%%%%%%%%%%
\subsection{Proof of Theorem \ref{mainthm}}

It remains to combine the convergence estimates of Theorems \ref{thm-cv-dirichlet} and \ref{thm-cv-neumann} to prove Theorem \ref{mainthm}. We use 
a slight modification of the superposition argument in \cite{CL18} in order to cope with the nonzero incoming condition, but we basically follow the same 
lines. Let us consider a final time $T \ge 1$ and some data $f \in H^{k+1}((0,L))$, $g \in H^{k+1}((0,T))$ that satisfy the compatibility conditions stated in 
Theorem \ref{mainthm}. We consider some function $\chi \in \mathcal{C}^\infty(\mathbb{R})$ such that $\chi(x)=0$ if $x \le 1/3$ and $\chi(x)=1$ if $x \ge 
2/3$. We then decompose the initial condition $f$ as:
$$
\forall \, x \in (0,L) \, ,\quad f(x) \, = \, (1-\chi(x/L)) \, f(x) +\chi(x/L) \, f(x) \, .
$$
Since $(1-\chi(\cdot/L)) \, f$ vanishes on $(2\, L/3,L)$, we can extend it by zero to the interval $(L,+\infty)$ and thus consider $(1-\chi(\cdot/L)) \, f$ as an 
element of $H^{k+1}((0,+\infty))$. Furthermore, the functions $(1-\chi(\cdot/L)) \, f$ and $g$ satisfy the same compatibility conditions as $f$ and $g$ at 
$t=x=0$. We can  thus apply Theorem \ref{thm-cv-dirichlet} to the sequence $(v_j^n)_{j \ge 1-r,0 \le n \le T/\Delta t}$ that is defined as the solution to the 
numerical scheme:
\begin{equation*}
\begin{cases}
v_j^0 \, = \, \dfrac{1}{\Delta x} \, {\displaystyle \int_{x_{j-1}}^{x_j}} (1-\chi(x/L)) \, f(x) \, {\rm d}x \, ,& j \ge 1 \, ,\\
v^n_\ell \, = \, {\displaystyle \sum_{\kappa=0}^{k-1}} \dfrac{\Delta x^\kappa}{(\kappa+1) \, ! \, (-a)^\kappa} \, \big( \ell^{\kappa+1} -(\ell-1)^{\kappa+1} \big) 
\, g^{(\kappa)}(t^n)  \, ,& 0 \le n \le T/\Delta t \, , \quad \ell=1-r,\dots,0 \, ,\\
v_j^{n+1} ={\displaystyle \sum_{\ell=-r}^p} a_\ell \, v_{j+\ell}^n \, ,& 0 \le n \le T/\Delta t-1 \, , \quad j \ge 1 \, .
\end{cases}
\end{equation*}
We obtain the estimate:
\begin{equation}
\label{erreur-vjn}
\sup_{0 \le n \le T/\Delta t} \, \left( \, \sum_{j \ge 1} \Delta x \, \left( v_j^n -\dfrac{1}{\Delta x} \, \int_{x_{j-1}}^{x_j} v(t^n,x) \, {\rm d}x \right)^2 \, \right)^{1/2} 
\, \le \, C \, T \, \, \Delta x^k \, \big( \Vert f \Vert_{H^{k+1}((0,L))} +\Vert g \Vert_{H^{k+1}((0,T))} \big) \, ,
\end{equation}
where $v$ is the exact solution to the transport problem:
\begin{equation*}
\begin{cases}
\partial_t v +a \, \partial_x v \, = \, 0 \, , & t \in (0,T) \, , \quad x \ge 0 \, ,\\
v(0,x) \, = \, (1-\chi(x/L)) \, f(x) \, , & x \ge 0 \, ,\\
v(t,0) \, = \, g(t) \, ,& t \in (0,T) \, .
\end{cases}
\end{equation*}
Similarly, we can view $\chi(\cdot/L) \, f$ as an element of $H^{k+1}((-\infty,L))$ that vanishes on $(-\infty,L/3)$. Theorem \ref{thm-cv-neumann} then shows 
that the solution $(w_j^n)_{j \le J+p,0 \le n \le T/\Delta t}$ to the numerical scheme:
\begin{equation*}
\begin{cases}
w_j^0 \, = \, \dfrac{1}{\Delta x} \, {\displaystyle \int_{x_{j-1}}^{x_j}} \chi(x/L) \, f(x) \, {\rm d}x \, ,& j \le J \, ,\\
(D_{-}^{k_b}w^n)_{J+\ell} \, = \, 0 \, ,& 0 \le n \le T/\Delta t \, , \quad \ell=1,\dots,p \, ,\\
w_j^{n+1} ={\displaystyle \sum_{\ell=-r}^p} a_\ell \, w_{j+\ell}^n \, ,& 0 \le n \le T/\Delta t-1 \, , \quad j \le J \, ,
\end{cases}
\end{equation*}
satisfies:
\begin{multline}
\label{erreur-wjn}
\sup_{0 \le n \le T/\Delta t} \, \left( \, \sum_{j \le J} \Delta x \, 
\left( w_j^n -\dfrac{1}{\Delta x} \, \int_{x_{j-1}}^{x_j} \chi((x-a \, t^n)/L) \, f(x-a \, t^n) \, {\rm d}x \right)^2 \, \right)^{1/2} \\
\le \, C \, T \, \Delta x^{\min (k,k_b)} \, \Vert f \Vert_{H^{k+1}((0,L))} \, .
\end{multline}

Using the support property of the function $\chi$ and the fact that the scheme \eqref{eq:scheme} is explicit with a finite stencil, we find that for all time iteration 
$n$ up to the threshold:
$$
N \, := \, \min \left( {\rm E} \left( \dfrac{J/3-k_b}{r} \right) \, , \, {\rm E} \left( \dfrac{{\rm E}(J/3)}{p} \right) \right) \, ,
$$
there holds:
$$
w_{1-r}^n \, = \, \cdots \, = \, w_0^n \, = \, 0 \, ,\quad v_{J+1-k_b}^n \, = \, \cdots \, = \, v_{J+p}^n \, = \, 0 \, .
$$
In particular, the solution $(u_j^n)_{1-r \le j \le J+p,0 \le n \le T/\Delta t}$ to \eqref{eq:scheme}, \eqref{initialdata}, \eqref{eq:bcNeumann}, \eqref{eq:bcDirichlet} 
satisfies:
$$
\forall \, n \, = \, 0,\dots,N \, ,\quad \forall \, j \, =1-r,\dots,J+p \, ,\quad u_j^n \, = \, v_j^n +w_j^n \, .
$$
Combining then the error estimates \eqref{erreur-vjn} and \eqref{erreur-wjn}, we obtain:
\begin{multline}
\label{erreur-ujn}
\sup_{0 \le n \le N} \, \left( \, \sum_{1 \le j \le J} \Delta x \, \left( u_j^n -\dfrac{1}{\Delta x} \, \int_{x_{j-1}}^{x_j} u(t^n,x) \, {\rm d}x \right)^2 \, \right)^{1/2} \\
\le \, C_1 \, T \, \, \Delta x^{\min (k,k_b)} \, \big( \Vert f \Vert_{H^{k+1}((0,L))} +\Vert g \Vert_{H^{k+1}((0,T))} \big) \, ,
\end{multline}
where $u$ is the exact solution to \eqref{TBE}.

It remains, as in \cite{CL18}, to iterate in time the error estimate \eqref{erreur-ujn}. We follow again the argument in \cite{CL18}. For any time iteration $n$ 
between $N$ and $2\, N$, we split the solution $(u_j^n)_{1-r \le j \le J+p,0 \le n \le T/\Delta t}$ to \eqref{eq:scheme}, \eqref{initialdata}, \eqref{eq:bcNeumann}, 
\eqref{eq:bcDirichlet} as the sum of the solution to the problem:
\begin{equation*}
\begin{cases}
\tilde{u}_j^N \, = \, \dfrac{1}{\Delta x} \, {\displaystyle \int_{x_{j-1}}^{x_j}} u(t^N,x) \, {\rm d}x \, ,& 1 \le j \le J \, ,\\
(D_{-}^{k_b} \tilde{u}^{N+n})_{J+\ell} \, = \, 0 \, ,& 0 \le n \le N \, , \quad \ell=1,\dots,p \, ,\\
\tilde{u}^{N+n}_\ell \, = \, {\displaystyle \sum_{\kappa=0}^{k-1} \dfrac{\Delta x^\kappa}{(\kappa+1) \, ! \, (-a)^\kappa}} \, 
\big( \ell^{\kappa+1} -(\ell-1)^{\kappa+1} \big) \, g^{(\kappa)}(t^{N+n}) \, ,& 0 \le n \le N \, , \quad \ell=1-r,\dots,0 \, ,\\
\tilde{u}_j^{N+n+1} ={\displaystyle \sum_{\ell=-r}^p} a_\ell \, \tilde{u}_{j+\ell}^{N+n} \, ,& 0 \le n \le N-1 \, , \quad 1 \le j \le J \, ,
\end{cases}
\end{equation*}
and of the (presumably small) solution to the `error' problem:
\begin{equation*}
\begin{cases}
\varepsilon_j^N \, = \, u_j^N -\dfrac{1}{\Delta x} \, {\displaystyle \int_{x_{j-1}}^{x_j}} u(t^N,x) \, {\rm d}x \, ,& 1 \le j \le J \, ,\\
(D_{-}^{k_b} \varepsilon^{N+n})_{J+\ell} \, = \, 0 \, ,& 0 \le n \le N \, , \quad \ell=1,\dots,p \, ,\\
\varepsilon^{N+n}_\ell \, = \, 0 \, ,& 0 \le n \le N \, , \quad \ell=1-r,\dots,0 \, ,\\
\varepsilon_j^{N+n+1} ={\displaystyle \sum_{\ell=-r}^p} a_\ell \, \varepsilon_{j+\ell}^{N+n} \, ,& 0 \le n \le N-1 \, , \quad 1 \le j \le J \, .
\end{cases}
\end{equation*}
Since the initial condition $u(\cdot-a\, t^N)$ and the boundary source term $g(t^N+\cdot)$ satisfy the compatibility conditions at the corner $t=x=0$, we 
can apply the first step of the proof (leading to the error estimate \eqref{erreur-ujn}) for the $(\tilde{u}_j^{N+n})$ part, and we apply the stability estimate 
of \cite[Proposition 4.1]{CL18} for the $(\varepsilon_j^{N+n})$ part. This leads to the second error estimate:
\begin{multline*}
\sup_{N \le n \le 2 \, N} \, \left( \, \sum_{1 \le j \le J} \Delta x \, \left( u_j^n -\dfrac{1}{\Delta x} \, \int_{x_{j-1}}^{x_j} u(t^n,x) \, {\rm d}x \right)^2 \, \right)^{1/2} \\
\le \, C_1 \, (1+C_2) \, T \, \, \Delta x^{\min (k,k_b)} \, \big( \Vert f \Vert_{H^{k+1}((0,L))} +\Vert g \Vert_{H^{k+1}((0,T))} \big) \, ,
\end{multline*}
and, more generally, to:
\begin{multline*}
\sup_{\mu \, N \le n \le (\mu+1) \, N} \, 
\left( \, \sum_{1 \le j \le J} \Delta x \, \left( u_j^n -\dfrac{1}{\Delta x} \, \int_{x_{j-1}}^{x_j} u(t^n,x) \, {\rm d}x \right)^2 \, \right)^{1/2} \\
\le \, C_1 \, \left( \sum_{\nu=0}^\mu C_2^\nu \right) \, T \, \, \Delta x^{\min (k,k_b)} \, \big( \Vert f \Vert_{H^{k+1}((0,L))} +\Vert g \Vert_{H^{k+1}((0,T))} \big) \, .
\end{multline*}
The end of the proof is the same as in \cite{CL18} and we refer the interested reader to that reference for the details.

%%%%%%%%%%%%%%%%%%%%%%%%%%%%%%%%%
\section{High order outflow boundary layer analysis}

In the present section, we explain how the analysis of \cite{BC17}, which dealt with the case of the Dirichlet boundary condition at the outflow boundary, can be 
extended to the case of high order extrapolation \eqref{eq:bcNeumann}. The goal is to obtain an accurate description of the numerical solution close to the outflow 
boundary by means of a boundary layer expansion. The leading order term in the expansion corresponds to the exact solution to the transport equation. However, 
this leading order term does not satisfy the extrapolation condition \eqref{eq:bcNeumann}, leading to a consistency error of magnitude $O(\Delta x^{k_b})$ on the 
boundary. Under some mild structural assumption on the numerical scheme \eqref{eq:scheme}, we show below that this $O(\Delta x^{k_b})$ error on the boundary 
gives rise to a boundary layer term which scales as $O(\Delta x^{k_b+1/2})$ in the $\ell^2_j$ norm. This gain of a factor $\Delta x^{1/2}$ enables us to recover the 
optimal convergence rate $k_b$ in the maximum norm on the whole spatial domain for $k_b<k$.

%%%%%%%%%%%%%%%%%%%%%%%%%%%%%%%%%%
\subsection{An introductive example}

Let us go back for a while to the case of the Lax-Wendroff scheme \eqref{LW}, which we consider here on the left half space:
$$
u_j^{n+1} \, = \, u_j^n - \dfrac{\lambda \, a}{2} \, (u_{j+1}^n-u_{j-1}^n) +\dfrac{(\lambda \, a)^2}{2} \, (u_{j+1}^n-2 \, u_j^n+u_{j-1}^n) \, ,\quad n \in \mathbb{N} 
\, ,\quad j \le J \, .
$$
At the outflow boundary, we impose the first order extrapolation condition (which corresponds to $k_b=1$ while $k=2$ for the Lax-Wendroff scheme):
$$
u_{J+1}^n \, = \, u_J^n \, ,\quad n \in \mathbb{N} \, .
$$
We start with some smooth initial condition $f$ defined on $(-\infty,L)$ which we project as a piecewise constant function:
$$
u_j^0 \, := \, \dfrac{1}{\Delta x} \, \int_{x_{j-1}}^{x_j} f(x) \, {\rm d}x \, ,\quad j \le J \, .
$$
The exact solution to the transport equation on $(-\infty,L)$ with initial condition $f$ is $u(t,x)=f(x-a\, t)$ (recall $a>0$). Hence the consistency analysis of the 
Lax-Wendroff scheme indicates that $u_j^n$ reads:
\begin{equation}
\label{LW-expansion1}
u_j^n \, = \, \dfrac{1}{\Delta x} \, \int_{x_{j-1}}^{x_j} f(x-a \, t^n) \, {\rm d}x +\text{\rm error} \, ,
\end{equation}
where the first term in the expansion on the right hand side yields an $O(\Delta x^2)$ consistency error in the interior domain, but also an $O(\Delta x)$ consistency 
error on the boundary. If we wish to push forward the above expansion, we need to take into account the boundary consistency error and introduce a corrector which 
will hopefully not alter the interior consistency error. This can be achieved by observing that  the sequence:
\[
v_j \, := \, \kappa^j \, ,\quad j \in \mathbb{Z} \, ,\qquad \kappa \, := \, -\dfrac{1+\lambda \, a}{1-\lambda \, a} \, ,
\]
is kept unchanged by the Lax-Wendroff scheme on $\mathbb{Z}$, and belongs to $\ell^2(-\infty,J)$ (we assume $0<\lambda \, a<1$ so $|\kappa|>1$). Hence, to 
remove the boundary consistency error, we can add a corrector on the right hand side of \eqref{LW-expansion1} in the following way:
\begin{equation}
\label{LW-expansion2}
u_j^n \, = \, \dfrac{1}{\Delta x} \, \int_{x_{j-1}}^{x_j} f(x-a \, t^n) \, {\rm d}x +\Delta x \, w^n \, v_{j-J} +\text{\rm error} \, ,
\end{equation}
where $w^n$ is defined in such a way that the two first terms on the right hand side satisfy the first order extrapolation condition, that is:
$$
w^n \, := \, -\dfrac{1}{\Delta x \, (\kappa-1)} \,  \int_{x_{J-1}}^{x_J} \dfrac{f(x+\Delta x-a \, t^n)-f(x-a \, t^n)}{\Delta x} \, {\rm d}x \, ,\quad n \in \mathbb{N} \, .
$$
If $f$ is sufficiently smooth, then $w^n$ is $O(1)$ and the first corrector on the right hand side of \eqref{LW-expansion2} is $O(\Delta x)$ in $\ell^\infty_j$. 
Note however that the $\ell^2$ norm (in space) of this boundary layer corrector scales as $\Delta x^{3/2}$ since the sequence $(\kappa^j)$ is square 
integrable on $(-\infty,0)$. Another important observation at this point is that defining $w^n$ requires the real number $\kappa$ not to equal $1$. This 
fact follows here from a mere verification but it is a general consequence of the analysis in \cite{goldberg} of the Lopatinskii determinant associated with 
the boundary condition \eqref{eq:bcNeumann} (see also the proof of Lemma \ref{lm:bndlayer} below).

At this stage, the error analysis amounts to studying the system satisfied by the sequence:
$$
\left( u_j^n -\dfrac{1}{\Delta x} \, \int_{x_{j-1}}^{x_j} f(x-a \, t^n) \, {\rm d}x -\Delta x \, w^n \, \kappa^{j-J} \right)_{j \le J+1, n \in \mathbb{N}} \, ,
$$
the main point being that there is no boundary forcing term, and since the added boundary layer corrector is $O(\Delta x^{3/2})$ in $\ell^2$, the initial 
condition and interior consistency errors will be $O(\Delta x^{3/2})$. Overall, the stability estimate for the Lax-Wendroff scheme with first order extrapolation 
at the boundary yields the convergence estimate:
$$
\sup_{n \le T/\Delta t} \, \left\| u_j^n -\dfrac{1}{\Delta x} \, \int_{x_{j-1}}^{x_j} f(x-a \, t^n) \, {\rm d}x -\Delta x \, w^n \, v_{j-J} \right\|_{\ell^2(-\infty,J)} \, = \, O(\Delta x^{3/2}) \, .
$$
By the triangle inequality, we thus obtain:
$$
\sup_{n \le T/\Delta t} \, \left\| u_j^n -\dfrac{1}{\Delta x} \, \int_{x_{j-1}}^{x_j} f(x-a \, t^n) \, {\rm d}x \right\|_{\ell^2(-\infty,J)} \, = \, O(\Delta x^{3/2}) \, ,
$$
and this immediately gives the uniform convergence estimate:
$$
\sup_{j \le J,n \le T/\Delta t} \, \left| u_j^n -\dfrac{1}{\Delta x} \, \int_{x_{j-1}}^{x_j} f(x-a \, t^n) \, {\rm d}x \right| \, = \, O(\Delta x) \, .
$$
The above brief sketch is made complete and rigorous below in the general framework of Theorem \ref{mainthm'}.

%%%%%%%%%%%%%%%%%%%%%%%%%%%%%%%%%%
\subsection{Discrete steady states}

Formalizing somehow the previous example in a more general framework, let us now introduce the following definition.

\begin{definition}[Steady state for the numerical scheme]
A sequence $(v_j)_{j \in \mathbb{Z}}$ is called a (discrete) steady state for the scheme \eqref{eq:scheme} if it is kept unchanged by the time 
iteration process on $\mathbb{Z}$, that is, if it satisfies:
\begin{equation}
\label{eq:characteristic'}
\forall \, j \in \mathbb{Z} \, ,\quad \sum_{\ell=-r}^p a_\ell \, v_{j+\ell} \, = \, v_j \, .
\end{equation}
\end{definition}

In order to characterize the discrete steady states, it is natural to introduce the characteristic polynomial:
\begin{equation}
\label{eq:characteristic}
A(X) \, := \, \sum_{\ell=-r}^p a_\ell \, X^{\ell+r} -X^r \, .
\end{equation}
From the consistency property \eqref{Assu1}, any constant sequence is a discrete steady state for \eqref{eq:scheme}, the same property being available 
for the continuous model (namely, the transport operator). However, the discrete nature of the differentiation operator involved in the numerical scheme 
\eqref{eq:scheme} allows the existence of many other discrete steady states. The latter play an important role when considering the half-space problem 
with some discrete boundary conditions.

From the non-characteristic assumption $a>0$, it follows that among the roots of $A$, $X=1$ is always a simple root. Let us now introduce the whole set 
of (pairwise distinct) roots of $A$ together with their multiplicities through the factorization of $A$ in $\mathbb{C}[X]$:
\begin{equation}
\label{eq:factorizationA}
A(X) \, = \, a_p \, \prod_{\sigma=1}^\tau \, (X-\kappa_\sigma)^{\mu_\sigma} \, .
\end{equation}
Clearly, looking at the degree of the polynomial $A$, one has the equality
\begin{equation*}
\sum_{\sigma=1}^{\tau} \mu_\sigma \, = \, r+p \, .
\end{equation*}
For convenience, we order the roots of $A$ with decreasing modulus:
\begin{equation*}
|\kappa_1| \, \geq \, |\kappa_2| \, \geq \, \cdots \, \geq \, |\kappa_{\tau}| \, .
\end{equation*}
To make the analysis more intelligible, we will work under the following assumption, which was already present in \cite{BC17}.

\begin{assumption}
\label{assumroot}
The characteristic polynomial $A$ defined in \eqref{eq:characteristic} has a unique root (equal to $1$) on the unit circle $\mathbb{S}^1 = \{ z \in \mathbb{C} 
\, , \, |z|=1 \}$. In other words, we assume:
\begin{equation}
\label{eq:assumroot}
\bigcup_{\sigma=1}^\tau \{ \kappa_\sigma \} \cap \mathbb{S} \, = \, \{ 1 \} \, .
\end{equation}
\end{assumption}

As observed on the above example of the Lax-Wendroff scheme, the steady states we are looking at should decrease rapidly as $j$ tends to $-\infty$, so 
that they provide with a localized correction (near the boundary) to the usual convergence analysis and belong to $\ell^2(-\infty,J)$. We are therefore only 
concerned with those roots of $A$ that have modulus larger than 1. Lemma \ref{lm:countproperty} below gives the precise number of such roots (counted 
with their multiplicities). We refer to \cite[Lemma 2.1]{BC17} for the proof.

\begin{lemma}[Unstable roots of $A$ \cite{BC17}]
\label{lm:countproperty}
Under assumptions \ref{ass:consistencystability} and \ref{assumroot}, letting $\kappa_1,\ldots,\kappa_{\tau_+}$ be the roots of $A$ that belong to 
$\mathbb{U} := \{ z \in \mathbb{C} \, , \, |z|>1 \}$ with their corresponding multiplicities $\mu_{1},\ldots,\mu_{\tau_+}$, then one has
\begin{equation}
\label{eq:countproperty}
\sum_{\sigma=1}^{\tau_+} \mu_{\sigma} \, = \, p \, .
\end{equation} 
\end{lemma}

\noindent A direct consequence of Lemma \ref{lm:countproperty} is the following description of steady states for \eqref{eq:scheme} that belong to 
$\ell^2(-\infty,J)$. The proof follows from the standard description of the set of solutions to the recurrence relation \eqref{eq:characteristic'}.

\begin{lemma}
\label{lm:ss}
The set of discrete steady states for the scheme \eqref{eq:scheme} that belong to $\ell^2(-\infty,J)$ is the finite dimensional linear subspace spanned 
by the $p$ linearly independent sequences $\rho^{(\sigma,\nu)}$:
\begin{equation}
\label{eq:generating}
\rho^{(\sigma,\nu)}_j \, := \, (j-J)^\nu \, \kappa_\sigma^{j-J} \, ,\quad j \in \mathbb{Z} \, ,\quad 1 \le \sigma \le \tau_+ \, ,\quad 0 \le \nu < \mu_\sigma \, .
\end{equation}
Equivalently, such discrete steady states in $\ell^2(-\infty,J)$ read:
\begin{equation}
\label{eq:sspolynb}
v_j \, = \, \sum_{\sigma=1}^{\tau_+} p_\sigma(j) \, \kappa_\sigma^{j-J} \, ,\quad j \in \mathbb{Z} \, ,
\end{equation}
where $p_\sigma\in\mathbb{C}_{\mu_\sigma-1}[X]$ for all index $1 \le \sigma \le \tau_+$.
\end{lemma}

Let us detail the parametrization of the set of (stable) discrete steady states on the two main examples we are concerned with. For the Lax-Wendroff 
scheme \eqref{LW}, one has:
$$
A(X) \, = \, -\dfrac{\lambda \, a \, (1-\lambda \, a)}{2} \, X^2 +(1-(\lambda \, a)^2) \, X +\dfrac{\lambda \, a \, (\lambda \, a+1)}{2} \, .
$$
The (two simple) roots of $A$ are $1$ and:
$$
\kappa \, := \, -\dfrac{1+\lambda \, a}{1-\lambda \, a} \, ,
$$
with $\kappa \in \mathbb{U}$ assuming, as usual, $0<\lambda \, a<1$. For the half space problem on $(-\infty,J)$, $\kappa$ is therefore the unique stable 
root, and $1$ counts as an unstable root (see \cite{BC17}). In particular, assumption \ref{assumroot} is satisfied. The set of solutions to \eqref{eq:characteristic'} 
that belong to $\ell^2(-\infty,J)$ is the one-dimensional subspace spanned by the sequence $(\kappa^{j-J})_{j \in \mathbb{Z}}$.

Let us now consider the so-called $O3$ scheme, which is a convex combination of the Lax-Wendroff and Beam-Warming schemes, see \cite{strang,despres}. 
We now have $p=1$ and $r=2$, and the scheme reads:
\begin{multline}
\label{eq:O3}
u_j^{n+1} \, = \, - \dfrac{\lambda \, a \, (1-(\lambda \, a)^2)}{6} \, u_{j-2}^n +\dfrac{\lambda \, a \, (1+\lambda \, a) \, (2-\lambda \, a)}{2} \, u_{j-1}^n \\
+\dfrac{(1-(\lambda \, a)^2) \, (2-\lambda \, a)}{2} \, u_j^n -\dfrac{\lambda \, a \, (1-\lambda \, a) \, (2-\lambda \, a)}{6} \, u_{j+1}^n \, ,
\end{multline}
with, again, $0<\lambda \, a<1$. Assumption \ref{ass:consistencystability} is then satisfied (with $k=3$). The roots of the corresponding characteristic polynomial 
$A$ are:
$$
\kappa_\pm \, := \, \dfrac{-(1+\lambda \, a) \, (5-2 \, \lambda \, a) \pm \sqrt{(1+\lambda \, a) \, (33- 15 \, \lambda \, a)}}{2 \, (1-\lambda \, a) \, (2-\lambda \, a)} 
\, ,\quad \kappa_0 \, := \, 1 \, ,
$$
each of them being simple. The root $\kappa_-$ is the only one in $\mathbb{U}$ and $\kappa_+$ belongs to the open unit disk $\mathbb{D}$, which is 
consistent with Lemma \ref{lm:ss} ($p=1$). In particular, assumption \ref{assumroot} is satisfied.

%%%%%%%%%%%%%%%%%%%%%%%%%%%%%%%%%%
\subsection{The boundary layer expansion. Proof of Theorem \ref{mainthm'}}

We now start proving Theorem \ref{mainthm'}, and for that, we consider some initial condition $f \in H^{k+1}((-\infty,L))$ which, for convenience, we extend 
to the whole real line $\mathbb{R}$ as an element of $H^{k+1}(\mathbb{R})$. Our aim is to compare the solution to the scheme \eqref{neumannsortie} 
(which is set on a half line) with the piecewise constant projection of the exact solution to the transport equation. We thus introduce the notation:
\begin{equation*}
\omega_j^n \, := \, \dfrac{1}{\Delta x} \, \int_{x_{j-1}}^{x_j} f(x-a \, t^n) \, \textrm{d}x \, ,\quad j \le J+p \, ,\quad n \in \mathbb{N} \, .
\end{equation*}
The consistency analysis in \cite{CL18} of the scheme \eqref{neumannsortie} amounts to considering the numerical scheme satisfied by the error 
$(u_j^n-\omega_j^n)$. It is proved in \cite{CL18} that the resulting boundary consistency errors have size $O(\Delta x^{k_b})$, while the interior 
consistency errors have size $O(\Delta x^k)$. Here we have $k_b<k$ so the worst term is on the boundary. Following the arguments in \cite{BC17}, 
we are therefore going to introduce a boundary layer corrector in order to remove the boundary consistency error, up to introducing new initial and 
interior consistency errors, whose size will be proven to be $O(\Delta x^{k_b+1/2})$ hence the final result of Theorem \ref{mainthm'}. Let us make 
this argument precise.

The consistent expansion of the numerical solution $(u_j^n)$ takes the form of a corrected version of $(\omega_j^n)$, involving now a boundary layer term 
$(v_j^n) \in \ell^2(-\infty,J)$ as for the above introductive example. The aim is to reduce the magnitude, at the boundary, of the following error:
\begin{equation}
\label{eq:modconv}
\varepsilon_j^n \, := \, \omega_j^n - u_j^n +\Delta x^{k_b} \, v_j^n \, ,\quad j \le J+p \, ,\quad n \in \mathbb{N} \, .
\end{equation}
The definition of $(v_j^n)_{j \le J+p,n\in\mathbb{N}}$ is chosen so as to correct the consistency error at the boundary. The simplest way to do so consists 
in chosing $(v^n_j)_{j \le J+p,n\in\mathbb{N}}$ so as to get precisely in the ghost cells the relations $(D_-^{k_b} \varepsilon^n)_{J+\ell} = 0$, $\ell=1,\dots,p$. 
From now on, we formulate the problem in such a way to normalize the generating sequences according to the value of $J$. In view of Lemma~\ref{lm:ss}, 
the problem to be solved writes:
\begin{align}
& v_j^n \, = \, \sum_{\sigma=1}^{\tau_+} \, \sum_{\nu=0}^{\mu_\sigma-1} z_{\sigma,\nu}^n \, \rho_j^{(\sigma,\nu)} \, ,\quad 
j \le J+p \, ,\quad n \in \mathbb{N} \, ,\label{eq:ss} \\
& (D_-^{k_b} v^n)_{J+\ell}  \, = \, - \dfrac{1}{\Delta x^{k_b}} \, (D_-^{k_b}\omega^n)_{J+\ell} \, ,\quad \ell=1,\dots,p \, ,\quad n \in \mathbb{N} \, ,\label{eq:bcequation}
\end{align}
where the sequences $\rho^{(\sigma,\nu)}$ are defined in \eqref{eq:generating}. 
Equivalently to \eqref{eq:ss}, we can look for the boundary layer corrector $(v_j^n)_{j\leq J+p,n\in\mathbb{N}}$ under the form
\begin{equation}
\label{eq:sspolynb'}
v_j^n \, = \, \sum_{\sigma=1}^{\tau_+} p_{n,\sigma} (j-J) \, \kappa_\sigma^{j-J} \, ,
\end{equation}
where $p_{n,\sigma} \in \mathbb{C}_{\mu_\sigma-1}[X]$ for all index $1 \le \sigma \le \tau_+$. The existence of the corrector $(v_j^n)$ is given by the following 
result. We recall that in the framework of Theorem \ref{mainthm'}, there holds $k_b<k$.

\begin{lemma}
\label{lm:bndlayer}
Consider the initial condition $f \in H^{k+1}((-\infty,L))$ extended to the whole real line $\mathbb{R}$. Then the boundary layer problem 
\eqref{eq:ss}-\eqref{eq:bcequation} admits a unique solution $(v_j^n)_{j \le J+p,n \in \mathbb{N}}$, and this solution satisfies the estimate:
\begin{equation}
\label{eq:sizebl}
\sup_{n \in \mathbb{N}} \, \left( \, \sum_{j \le J} \Delta x \, (v_j^n)^2 \, \right)^{1/2} \le C \, \Delta x^{1/2} \, \| f \|_{H^{k_b+1}((-\infty,L))} \, ,
\end{equation}
where the constant $C>0$ is independent of $\Delta x>0$, $J$, $L$ and $f$.
\end{lemma}

\begin{proof} Let us fix some integer $n\in\mathbb{N}$. The solution $(u_j^n)_{j \le J+p}$ to \eqref{neumannsortie} solves the homogeneous boundary condition 
\eqref{eq:bcNeumann}, thus equivalently to \eqref{eq:bcequation} one has to find the vector of coordinates $z \in \mathbb{C}^p$ solution to the linear system 
$A_{k_b} \, z+b=0$ where $b=\Delta x^{-k_b} \, ((D_-^{k_b} \omega^n)_{J+\ell})_{1 \le \ell \le p}$, and the $p \times p$ matrix $A_{k_b}$ is defined as follows:
\begin{equation*}
A_{k_b} \, := \, \begin{pmatrix}
(D_-^{k_b} \rho^{(1)})_{1} & \ldots & (D_-^{k_b} \rho^{(p)})_{1} \\
\vdots & & \vdots \\
(D_-^{k_b} \rho^{(1)})_{p} & \ldots & (D_-^{k_b} \rho^{(p)})_{p}
\end{pmatrix} \, ,
\end{equation*}
where we have relabeled the sequences $\rho^{(\sigma,\nu)}$, $\sigma=1,\dots,\tau_+$, $\nu=0,\dots,\mu_\sigma-1$ as $\rho^{(1)},\dots,\rho^{(p)}$ in order to 
make the definition of $A_{k_b}$ easier to read. The latter matrix is somehow the $k_b$th-order discrete derivative of the so-called confluent Vandermonde matrix. 
It seems possible to compute the determinant of $A_0$, see~\cite{HornJohnson94}, and then to extend this result to higher values of $k_b$ but we prefer to avoid 
such complicated computations. From the identity of dimensions, we shall just prove that the matrix $A_{k_b}$ is one-to-one, in other words we shall prove that the 
problem \eqref{eq:ss}-\eqref{eq:bcequation}, or equivalently \eqref{eq:sspolynb'}-\eqref{eq:bcequation}, admits a trivial kernel.

Dealing with discrete derivatives of products of polynomial and/or geometric sequences, the divided difference algebra appears as a suitable tool in our analysis. 
For more details we refer the interested reader to \cite{Steffensen39,Popoviciu40,deBoor05}. For consistency in the notation, we recall hereafter the recursive 
definition of divided differences, but specified for the case of consecutive integer abscissae. Being given a sequence of complex numbers $(w_j)_{j\in\mathbb{Z}}$, 
one defines:
\begin{equation}
\label{eq:recdiffdiv}
\begin{aligned}
& w[j] \, := \, w_j \, ,\quad j \in \mathbb{Z} \, ,\\
& w[j-k,\ldots,j] \, := \, \dfrac{1}{m} \, \Bigl( w[j-m+1,\ldots,j] - w[j-m,\ldots,j-1]\Bigr) \, ,\quad j \in \mathbb{Z} \, , \, m \in \mathbb{N}^\star \, .
\end{aligned}
\end{equation}
The quantity $(D_-^{k_b}w)_j$ is directly related to the divided difference $w[j-k_b,\ldots,j]$ by the equality:
\begin{equation}
\label{eq:diffdiv}
(D_-^{k_b} w)_j \, = \, k_b! \, w[j-k_b,\ldots,j] \, ,\quad j \in \mathbb{Z} \, .
\end{equation}
Importantly, we may also use the Leibniz formula for divided differences of a product of two sequences:
\begin{equation}
\label{eq:leibniz}
(w \, \tilde{w})[j-k_b,\ldots,j] \, = \, \sum_{m=0}^{k_b} w[j-k_b,\ldots ,j-m] \, \tilde{w}[j-m,\ldots,j] \, ,\quad j \in \mathbb{Z} \, .
\end{equation}
In terms of the $D_-$ operator, using the relation \eqref{eq:diffdiv}, the Leibniz formula \eqref{eq:leibniz} rewrites under the more recognizable form:
\begin{equation*}
(D_-^{k_b}(w \, \tilde{w}))_j \, = \, \sum_{m=0}^{k_b} {k_b \choose m} \, (D_-^{k_b-m} w)_{j-m} \, (D_-^m \tilde{w})_j \, ,\quad j \in \mathbb{Z} \, .
\end{equation*}

Let us continue with the representation formula \eqref{eq:sspolynb} of the solution to the boundary layer problem. Looking at the kernel of the linear problem 
\eqref{eq:bcequation}, we have to find polynomials $(p_\sigma)_{1 \le \sigma \le \tau_+}$ with respective degrees less than or equal to $(\mu_\sigma-1)_{1 \le 
\sigma \le \tau_+}$, satisfying the set of equations:
\[
\sum_{\sigma=1}^{\tau_+} \, \sum_{m=0}^{k_b} \, p_\sigma[\ell-k_b,\ldots, \ell-m] \, \kappa_\sigma[\ell-m,\ldots, \ell] \, = \, 0 \, ,\quad 1 \le \ell \le p \, ,
\]
where we denote, with a slight abuse in the notation, $\kappa_\sigma$ for the corresponding geometric sequence $(\kappa_m)_{m \in \mathbb{Z}}$, for any 
$\sigma=1,\dots,\tau_+$. Actually, from the identity \eqref{eq:diffdiv} and by induction on the integer $m$ (or using \eqref{eq:recdiffdiv}), it is easy to prove 
that the $m$-th order divided difference of $\kappa_\sigma$ is given by:
\[
\kappa_\sigma [\ell-m,\ldots ,\ell] \, = \, \dfrac{1}{m!} \, (D_-^m \kappa_\sigma)_\ell \, = \, \dfrac{1}{m!} \, (1-\kappa_\sigma^{-1})^m \, \kappa_\sigma^\ell \, ,
\quad 1 \le \ell \le p \, .
\]
Let us introduce, for any integer $\sigma$ and any polynomial $p_\sigma$ with degree less than or equal to $\mu_\sigma-1$, the following polynomial 
$Q_\sigma$ also with degree less than or equal to $\mu_\sigma-1$:
\begin{equation}
\label{eq:changepolyn}
Q_\sigma(X) \, := \, \sum_{k=0}^{k_b} \, (1-\kappa_\sigma^{-1})^k \, p_\sigma[X-k_b,\ldots ,X-k] \, .
\end{equation}
With these notations, the equations to solve now equivalently read:
\begin{equation}
\label{eq:genVDM}
\sum_{\sigma=1}^{\tau_+} Q_\sigma(\ell) \, \kappa_\sigma^\ell \, = \, 0,\quad 1 \le \ell \le p \, .
\end{equation}
Actually, the above set of equations \eqref{eq:genVDM} exactly corresponds to the generalized Lagrange-Hermite interpolation problem, which is known 
to be invertible. Thus one necessarily has $Q_\sigma=0$ for any $\sigma=1,\dots,\tau_+$. It then only remains to deduce that any of the polynomials 
$p_{\sigma}$ is also zero.

Observe that for any integer $k$ with $0 \le k \le k_b$, from the divided difference algebra, the polynomial $p_\sigma[X-k_b,\ldots ,X-k]$ has degree less 
than $\mu_\sigma-(k_b-k)$, see \eqref{eq:diffdiv}. Thus the highest degree polynomial involved in the sum \eqref{eq:changepolyn} is $p_\sigma[X-k_b]$ 
(for $k=k_b$). Since we know that $Q_\sigma$ is zero, then $p_\sigma$ is also necessarily zero (consider the highest degree coefficient). The injectivity 
of the boundary layer problem \eqref{eq:sspolynb'}-\eqref{eq:bcequation} is proved, and the matrix $A_{k_b}$ is therefore invertible.

As a consequence, there exist some uniquely determined coefficients $(\beta_{\sigma,\nu,\ell})$ that depend only on the considered scheme and on the 
extrapolation order $k_b$ (but neither on the initial condition $f$ nor on the time index $n$), such that the solution to \eqref{eq:ss}-\eqref{eq:bcequation} 
has the form:
\begin{equation}
\label{eq:solvez}
v_j^n \, = \, \Delta x^{-k_b} \, \sum_{\sigma=1}^{\tau_+} \, \sum_{\nu=0}^{\mu_\sigma-1} \, \sum_{\ell=1}^p \beta_{\sigma,\nu,\ell} \, 
(D_-^{k_b} \omega^n)_{J+\ell} \, \rho_j^{(\sigma,\nu)} \, .
\end{equation}
Using now triangular inequalities, we obtain, for some constant $C>0$, the upper bound:
\[
(v_j^n)^2 \, \le \, C \, \Delta x^{-2 \, k_b} \, \sum_{\ell=1}^p \big( (D_-^{k_b}\omega^n)_{J+\ell} \big)^2 \, 
\sum_{\sigma=1}^{\tau_+} \, \sum_{\nu=0}^{\mu_\sigma-1} \, (\rho_j^{(\sigma,\nu)})^2 \, ,\quad j \le J \, .
\]
On the one side, we recall the definition \eqref{eq:generating} of the sequences $\rho^{(\sigma,\nu)}$ in Lemma \ref{lm:ss}, hence the estimate:
\begin{equation}
\label{eq:normrho}
\left( \sum_{j \le J} \Delta x \, \sum_{\sigma=1}^{\tau_+} \, \sum_{\nu=0}^{\mu_\sigma-1} \, (\rho_j^{(\sigma,\nu)})^2 \right)^{1/2} \, \le \, C \, \sqrt{\Delta x} \, ,
\end{equation}
where the constant $C>0$ is independent of $J$ and $\Delta x$. On the other side, from \cite[Lemma 3.6]{CL18} and the continuity of the reflection operator 
from $H^{k_b+1}((-\infty,L))$ to $H^{k_b+1}(\mathbb{R})$, we have the upper bound:
\[
\big| (D_-^{k_b} \omega^n)_{J+\ell} \big| \, \le \, C \, \Delta x^{k_b} \, \| f \|_{H^{k_b+1}((-\infty,L))} \, ,\quad \ell=1,\dots,p \, ,\quad n \in \mathbb{N} \, ,
\]
and thus the required estimate~\eqref{eq:sizebl} follows.
\end{proof}

The interested reader will find in \cite{goldberg} a similar argument to the one developed in the proof of Lemma \ref{lm:bndlayer}. In \cite{goldberg}, the 
analysis of the determinant of the matrix $A_{k_b}$ arises from the verification of the so-called Uniform Kreiss-Lopatinskii Condition (a condition whose 
significance is based on the work \cite{gks}). Let us now prove Theorem~\ref{mainthm'}. The error $(\varepsilon_j^n)_{j \le J+p,n \in \mathbb{N}}$ introduced 
in \eqref{eq:modconv}, and fully defined through Lemma~\ref{lm:bndlayer}, satisfies the following set of equations\footnote{Here we use $u_j^0=\omega_j^0$ 
for $j \le J$.}:
\begin{equation}
\label{eq:schemehomogeneousBC}
\left\{
\begin{aligned}
& \varepsilon_j^0 \, = \, \Delta x^{k_b} \, v_j^0 \, , & & j \le J \, ,\\
& (D_-^{k_b} \varepsilon^n)_{J+\ell} \, = \, 0 \, , & & 0 \le n \le T/\Delta t \, ,\quad \ell=1,\ldots,p \, ,\\
& \varepsilon_j^{n+1} \, = \, \sum_{\ell=-r}^p a_\ell \, \varepsilon_{j+\ell}^n +\Delta t \, F_j^{n+1} \, , & & 0 \le n \le T/\Delta t-1 \, ,\quad  j \le J \, .
\end{aligned}
\right.
\end{equation}
Here above, the consistency error $F_j^{n+1}$ consists of two terms: a first one coming from the usual interior consistency error denoted $e_j^{n+1}$, and a 
second one coming from the time evolution of the boundary layer corrector denoted $\delta_j^{n+1}$. In other words, we split $F_j^{n+1} = e_j^{n+1}+\delta_j^{n+1}$ 
with:
\begin{equation*}
e_j^{n+1} \, := \, \dfrac{1}{\Delta t} \, \left( \omega_j^{n+1} -\sum_{\ell=-r}^p a_\ell \, \omega_{j+\ell}^n \right) \, ,\quad \textup{and} \quad 
\delta_j^{n+1} \, := \, \dfrac{\Delta x^{k_b}}{\Delta t} \, \left( v_j^{n+1} -\sum_{\ell=-r}^p a_\ell \, v_{j+\ell}^n \right) \, .
\end{equation*}
Considering the scheme \eqref{eq:schemehomogeneousBC}, the error $(\varepsilon_j^n)_{j \le J+p,0 \le n \le T/\Delta t}$ obeys the stability estimate 
applicable in the case of the homogeneous extrapolation boundary condition, see \cite[Proposition~3.4]{CL18}:
\begin{equation}
\label{eq:stabilityhomogeneous}
\sup_{0 \le n \le T/\Delta t} \sum_{j \le J} \Delta x \, (\varepsilon_j^n)^2 \, \le \, C \, \left\{ \sum_{j \le J} \Delta x \, (\varepsilon_j^0)^2 
+T^2 \, \sup_{1 \le n \le T/\Delta t} \, \sum_{j \le J} \Delta x \, (F_j^n)^2 \right\} \, .
\end{equation}
It therefore remains to estimate the initial and interior consistency errors in \eqref{eq:schemehomogeneousBC}:
\begin{itemize}
   \item \underline{The initial consistency error}. Estimating the initial condition $(\varepsilon_j^0)_{j \le J}$ directly follows from the estimate \eqref{eq:sizebl} 
   in Lemma \ref{lm:bndlayer}:
$$
\sum_{j \le J} \Delta x \, (\varepsilon_j^0)^2 \, \le \, C \, \Delta x ^{2 \, k_b+1} \, \| f \|^2_{H^{k_b+1}((-\infty,L))} \, .
$$
   \item \underline{The interior consistency error. I}. Estimating the interior consistency error $(e_j^n)$ related to the projected exact solution $(\omega_j^n)$ 
   has already been achieved in \cite{CL18} so we just report the result:
$$
\sup_{1 \le n \le T/\Delta t} \, \sum_{j \le J} \Delta x \, (e_j^n)^2 \, \le \, C \, \Delta x^{2\, k} \, \| f \|^2_{H^{k+1}((-\infty,L))} \, .
$$
   \item \underline{The interior consistency error. II}. Estimation of the new error term related to $(\delta_j^n)$. Observe that, first due to the steady states 
   decomposition from Lemma \ref{lm:ss}, and then using successively \eqref{eq:ss} and \eqref{eq:solvez}, the interior consistency error arising from the 
   boundary layer corrector rewrites as:
\begin{equation*}
\delta_j^{n+1} \, = \, \dfrac{\Delta x^{k_b}}{\Delta t} \, (v_j^{n+1}-v_j^{n}) \, = \, \dfrac{1}{\Delta t} \, 
\sum_{\sigma=1}^{\tau_+} \, \sum_{\nu=0}^{\mu_\sigma-1} \, \sum_{\ell=1}^p \beta_{\sigma,\nu,\ell} \, (D_-^{k_b}(\omega^{n+1}-\omega^n))_{J+\ell} \, 
\rho_j^{(\sigma,\nu)} \, .
\end{equation*}
Thus, from Cauchy-Schwarz inequalities, there exists a constant $C$ such that
\begin{equation*}
\sum_{j \le J} \Delta x \, (\delta_j^{n+1})^2 \, \le \, C \, \Delta x \, \sup_{\ell=1,\dots,p} \left( D_-^{k_b}\left(\dfrac{\omega^{n+1}-\omega^n}{\Delta t} \right)_{J+\ell} 
\right)^2 \, .
\end{equation*}
In the above formula, the discrete in time derivative of $\omega_j^n$ rewrites, for any $j \le J+p$ as
\begin{align*}
\dfrac{\omega^{n+1}_j-\omega^n_j}{\Delta t} \, =& \, \dfrac{1}{\Delta x} \, \int_{x_{j-1}}^{x_j} \dfrac{f(x-at^n-a\Delta t)-f(x-at^n)}{\Delta t} \, \textrm{d}x \\
=& \, \dfrac{1}{\Delta x}\int_{x_{j-1}}^{x_j} \underbrace{\dfrac{1}{\Delta t} \, \int_{x-at^n}^{x-at^n-a\Delta t} f'(y)\,\textrm{d}y}_{=:F(x)} \, \textrm{d}x \, .
\end{align*}
Since $f \in H^{k+1}(\mathbb{R})$ with $k>k_b$, we have at least $f \in H^{k_b+2}(\mathbb{R})$ and therefore $F \in H^{k_b+1}(\mathbb{R})$ with
\[
\| F^{(k_b+1)} \|_{L^2(\mathbb{R})} \, \le \, a^2 \, \| f^{(k_b+2)} \|_{L^2(\mathbb{R})} \, ,
\]
from which we deduce, using again \cite[Lemma 3.6]{CL18}:
\[
\left| D_-^{k_b} \left( \dfrac{\omega^{n+1}-\omega^n}{\Delta t} \right)_{J+\ell} \right| \, \le \, C \, \Delta x^{k_b} \, \| f \|_{H^{k_b+2}(\mathbb{R})} \, ,\quad 
\ell=1,\dots,p \, .
\]
Thus, using again the upper bound~\eqref{eq:normrho}, the above estimate and the $H^{k_b+2}$-continuity of the extension operator, we end up with:
$$
\sup_{1 \le n \le T/\Delta t} \, \sum_{j \le J} \Delta x \, (\delta_j^n)^2 \, \le \, C \, \Delta x^{2 \, k_b+1} \, \| f \|^2_{H^{k+1}((-\infty,L))} \, .
$$
\end{itemize}

Let us now come back to the stability estimate~\eqref{eq:stabilityhomogeneous} and use the three above consistency estimates to get (recall 
$T \ge 1$ and $k_b<k$):
$$
\sup_{0 \le n \le T/\Delta t} \, \left( \sum_{j \le J} \Delta x \, (\varepsilon_j^n)^2 \right)^{1/2} \, \le C \, T \, \Delta x ^{k_b+1/2} \, \| f \|_{H^{k+1}((-\infty,L))} \, .
$$
From the constructive formula for the boundary layer corrector $(v_j^n)$, we have derived the bound \eqref{eq:sizebl} which, by the triangle 
inequality, yields the convergence estimate (recall $\varepsilon_j^n=\omega_j^n-u_j^n+\Delta x ^{k_b} \, v_j^n$):
$$
\sup_{0 \le n \le T/\Delta t} \, \left( \sum_{j \le J} \Delta x \, (u_j^n -\omega_j^n)^2 \right)^{1/2} \, \le C \, T \, \Delta x ^{k_b+1/2} \, \| f \|_{H^{k+1}} \, .
$$
Using now the (crude) estimate:
$$
\sup_{j \le J} \, |b_j| \, \le \, \Delta x^{-1/2} \, \left( \sum_{j \le J} \Delta x \, b_j^2 \right)^{1/2} \, ,
$$
we complete the proof of Theorem \ref{mainthm'}.

%%%%%%%%%%%%%%%%%%%%%%%%%%%%%%%%%
\section{Numerical experiments}

%%%%%%%%%%%%%%%%%%%%%%%%%%%%%%%%%
\subsection{The Lax-Wendroff scheme}

We report in this paragraph on various numerical experiments with the Lax-Wendroff scheme \eqref{LW} (which corresponds to $p=r=1$). Assumption 
\ref{ass:consistencystability} is satisfied provided that $\lambda \, a \le 1$, and the order $k$ equals $2$. In all what follows, we choose $a=1$ and 
$\lambda=5/6$. The interval length is $L=6$ and the final time $T$ equals $8$. The initial condition is $f(x)=\sin x$ and the boundary source term is 
$g(t)=-\sin t$ so that the exact solution to \eqref{TBE} is $u(t,x)=\sin (x-t)$. With the values of $J$ reported in Table \ref{Table1} below, we implement 
the Lax-Wendroff scheme \eqref{LW} with the following numerical boundary conditions:
\begin{align*}
&u_{J+1}^n \, = \, u_J^n \, , \quad \text{\rm (first order outflow extrapolation condition),} \\
&u_0^n \, = \, \begin{cases}
-\sin t^n \, ,& \text{\rm (Dirichlet inflow condition \eqref{LW-inflow}),} \\
-\sin t^n -(\Delta x/2) \, \cos t^n \, ,& \text{\rm (inverse Lax-Wendroff inflow condition \eqref{eq:bcDirichlet}).} 
\end{cases}
\end{align*}
The errors, as measured in the statement of Theorem \ref{mainthm}, are reported in Table \ref{Table1} below for each of the two cases (either the 
Dirichlet inflow condition \eqref{LW-inflow} or the inverse Lax-Wendroff inflow condition \eqref{eq:bcDirichlet}). In either case, the observed convergence 
rate is $1$ since increasing $J$ by a factor $2$ decreases the error of the same factor $2$. This behavior is fully justified by Theorem \ref{mainthm'} 
since we have $k_b<k$ here.

%--------------- Table 1--------------------------------
\begin{table}[!htp]
\centering
\begin{tabular}{|c|c|c|}
\hline 
Number of cells $J$ & Dirichlet inflow condition & Inverse Lax-Wendroff inflow condition \\
\hline 
1000 & $4.1 \cdot 10^{-3}$ & $5.1 \cdot 10^{-4}$ \\
\hline 
2000 & $2.1 \cdot 10^{-3}$ & $2.5 \cdot 10^{-4}$ \\
\hline 
4000 & $1.1 \cdot 10^{-3}$ & $1.3 \cdot 10^{-4}$ \\
\hline 
8000 & $5.3 \cdot 10^{-4}$ &$6.3 \cdot 10^{-5}$ \\
\hline 
\end{tabular}
\caption{The $\ell_{n,j}^\infty$ error for the Lax-Wendroff scheme \eqref{LW} with first order outflow extrapolation and either the Dirichlet, or inverse 
Lax-Wendroff, inflow condition.}
\label{Table1}
\end{table}

We now turn to the second order outflow extrapolation condition:
\begin{align*}
&u_{J+1}^n \, = \, 2 \, u_J^n-u_{J-1}^n \, , \quad \text{\rm (second order outflow extrapolation condition \eqref{LW-outflow}),} \\
&u_0^n \, = \, \begin{cases}
-\sin t^n \, ,& \text{\rm (Dirichlet inflow condition \eqref{LW-inflow}),} \\
-\sin t^n -(\Delta x/2) \, \cos t^n \, ,& \text{\rm (inverse Lax-Wendroff inflow condition \eqref{eq:bcDirichlet}).} 
\end{cases}
\end{align*}
The errors, as measured in the statement of Theorem \ref{mainthm}, are reported in Table \ref{Table2} below for each of the two cases (either the 
Dirichlet inflow condition \eqref{LW-inflow} or the inverse Lax-Wendroff inflow condition \eqref{eq:bcDirichlet}). For the Dirichlet inflow condition, the 
observed convergence rate is $1$ again (despite the more accurate outflow treatment), but one recovers the convergence rate $2$ with the inverse 
Lax-Wendroff inflow condition \eqref{eq:bcDirichlet}. However, proving rigorously that this numerical scheme converges with the rate $2$ in the 
maximum norm might be very difficult (it might actually even be wrong !), even for smooth data, since the Lax-Wendroff scheme is known to be 
unstable in $\ell^\infty (\mathbb{Z})$. Improving the convergence rate $3/2$ of Theorem \ref{mainthm} in the case of the Lax-Wendroff scheme with 
second order extrapolation outflow condition is left to a future work.

%--------------- Table 2--------------------------------
\begin{table}[!htp]
\centering
\begin{tabular}{|c|c|c|}
\hline 
Number of cells $J$ & Dirichlet inflow condition & Inverse Lax-Wendroff inflow condition \\
\hline 
1000 & $3.7 \cdot 10^{-3}$ & $1.2 \cdot 10^{-5}$ \\
\hline 
2000 & $1.8 \cdot 10^{-3}$ & $2.9 \cdot 10^{-6}$ \\
\hline 
4000 & $9.3 \cdot 10^{-4}$ & $7.3 \cdot 10^{-7}$ \\
\hline 
8000 & $4.7 \cdot 10^{-4}$ & $1.8 \cdot 10^{-7}$ \\
\hline 
\end{tabular}
\caption{The $\ell_{n,j}^\infty$ error for the Lax-Wendroff scheme \eqref{LW} with second order outflow extrapolation \eqref{LW-outflow} 
and either the Dirichlet or inverse Lax-Wendroff inflow condition.}
\label{Table2}
\end{table}

%%%%%%%%%%%%%%%%%%%%%%%%%%%%%%%%%
\subsection{The $O3$ scheme}

Let us now consider the $O3$ scheme \eqref{eq:O3}, which is implemented by considering the recurrence:
$$
u_j^{n+1} \, = \, a_{-2} \, u_{j-2}^n +a_{-1} \, u_{j-1}^n +a_0 \, u_j^n +a_1 \, u_{j+1}^n \, , \quad n \in \mathbb{N}\, ,\quad j=1,\dots,J \, ,
$$
with:
\begin{align*}
&a_{-2} \, := \, -\dfrac{\lambda \, a}{6} \, \big( 1-(\lambda \, a)^2 \big) \, ,\quad a_{-1} \, := \, \dfrac{\lambda \, a}{2} \, (1+\lambda \, a) \, (2-\lambda \, a) \, ,\\
&a_0 \, := \, \dfrac{1}{2} \, \big( 1-(\lambda \, a)^2 \big) \, (2-\lambda \, a) \, ,\quad a_1 \, := \, -\dfrac{\lambda \, a}{6} \, (1-\lambda \, a) \, (2-\lambda \, a) \, .
\end{align*}
The reader can verify that assumption \ref{ass:consistencystability} is satisfied provided that $\lambda \, a \le 1$, and the order $k$ equals $3$ 
($r=2$ and $p=1$ here). To maintain third order accuracy, we implement the latter scheme with the following boundary conditions:
\begin{align*}
&u_{J+1}^n \, = \, 3 \, u_J^n -3 \, u_{J-1}^n +u_{J-2}^n \, , \quad \text{\rm (third order outflow extrapolation condition, $k_b=3$),} \\
&u_0^n \, = \, -\sin t^n -(\Delta x/2) \, \cos t^n +(\Delta x^2/6) \, \sin t^n\, , \quad \text{\rm (inverse Lax-Wendroff inflow condition \eqref{eq:bcDirichlet}),} \\
&u_{-1}^n \, = \, -\sin t^n -(3 \, \Delta x/2) \, \cos t^n +(7 \, \Delta x^2/6) \, \sin t^n\, , 
\quad \text{\rm (inverse Lax-Wendroff inflow condition \eqref{eq:bcDirichlet}).} 
\end{align*}
The measured errors are reported in Table \ref{Table3} below. They correspond to a rate of convergence $3$. Let us observe that the $O3$ scheme 
is known to be stable in $\ell^\infty(\mathbb{Z})$, see \cite{thomee,despres}, hence there is a genuine hope of proving rigorously that this rate of 
convergence does indeed hold (for smooth compatible data). Such a justification is also left to a future work.

%--------------- Table 3--------------------------------
\begin{table}[!htp]
\centering
\begin{tabular}{|c|c|c|}
\hline 
Number of cells $J$ & Inverse Lax-Wendroff inflow condition \\ 
\hline 
1000 & $2.1 \cdot 10^{-8}$ \\
\hline 
2000 & $2.6 \cdot 10^{-9}$ \\
\hline 
4000 & $3.3 \cdot 10^{-10}$ \\
\hline 
\end{tabular}
\caption{The $\ell_{n,j}^\infty$ error for the $O3$ scheme \eqref{LW} with third order outflow extrapolation and the inverse Lax-Wendroff inflow condition.}
\label{Table3}
\end{table}

%%%%%%%%%%%%%%%%%%%%%%%%%%%%%%%%%
%Bibliographie
\bibliographystyle{alpha}
\bibliography{Nabuco19}
\end{document}